\documentclass[11pt,a4paper,reqno]{amsart}

\usepackage{lmodern,amsmath,amssymb,amsthm}
\usepackage{fontspec}

\usepackage{microtype,needspace}
\usepackage[a4paper,margin=28mm]{geometry}

\usepackage{tikz,tikz-cd}
\usetikzlibrary{decorations.pathreplacing}

\usepackage[numbers]{natbib}
\usepackage[hidelinks]{hyperref}
\newcommand{\kk}{k}

\hypersetup{
pdftitle={Non-orbit 3-representation-finite symmetric algebras},
pdfauthor={Tor Kringeland},
pdfsubject={Higher Auslander--Reiten theory and orbit algebras}
}

\theoremstyle{definition}
\newtheorem{theorem}{Theorem}
\newtheorem{maintheorem}{Theorem}

\newtheorem{proposition}[theorem]{Proposition}

\DeclareMathOperator{\Ext}{Ext}
\DeclareMathOperator{\rad}{rad}
\DeclareMathOperator{\add}{add}

\DeclareMathOperator{\tr}{tr}

\DeclareMathOperator{\Hom}{Hom}
\DeclareMathOperator{\End}{End}
\DeclareMathOperator{\gldim}{gl.dim}
\newcommand{\rep}{\widehat{\Lambda}}
\newcommand{\Nak}{\widehat{\nu}}

\newcommand{\ZZ}{\mathbb{Z}}
\newcommand{\QQ}{\mathbb{Q}}
\newcommand{\dual}{\mathbb D}
\newcommand{\setof}[1]{\{\,#1\,\}}

\title[Non-orbit $3$-representation-finite symmetric algebras]
{Non-orbit $3$-representation-finite\\ symmetric algebras}

\author{Tor Kringeland}
\address{Department of Mathematical Sciences, NTNU, 7491 Trondheim, Norway}
\date{}

\begin{document}

\begin{abstract}
  We construct two $3$-representation-finite symmetric algebras over
  an algebraically closed field of characteristic zero that are not
  orbit algebras of repetitive categories.  This answers a question of
  Darpö and Iyama in the negative.
\end{abstract}

\maketitle

\section{Introduction}

For $d\geq2$ a module $M$ over a finite-dimensional algebra is
$d$-cluster-tilting if
\begin{equation*}
  \add M = \setof{X : \Ext^{i}(M,X)=0 \text{ for } 0<i<d}
  = \setof{X : \Ext^{i}(X,M)=0 \text{ for } 0<i<d},
\end{equation*}
and the algebra is $d$-representation-finite if it possesses such a
module.  \citet{DI} construct $d$-representation-finite self-injective
algebras as orbit algebras of repetitive categories of
finite-dimensional algebras under suitable automorphisms.  Their orbit
construction generalizes \citeauthor{Riedtmann}'s description of the
self-injective algebras of finite representation
type~\cite{Riedtmann}.  They ask whether every
$d$-representation-finite self-injective algebra over an algebraically
closed field of characteristic zero arises from this construction.

We give two counterexamples, symmetric algebras that are
$3$-representation-finite but do not arise from the orbit
construction.  Over such a field $\kk$ let $A$ be the algebra with the
quiver and relations of the quaternion-type algebra
$Q(3\mathcal{A})^{2}_{2}$ and let $S$ be \citeauthor{Erdmann}'s
$3$-spherical weighted surface algebra with all multiplicities and
parameters equal to $1$~\cite{Erdmann}.

\begin{maintheorem}\label{thm:main}
  The algebras $A$ and $S$ are $3$-representation-finite.  Neither can
  be realized as an orbit algebra of the repetitive category of a
  finite-dimensional algebra.
\end{maintheorem}

For the first assertion we add one or two summands to a module
considered by \citeauthor{Erdmann} and check Iyama's criterion by
computer.  For the second our trick is to compare traces of
derivations coming from the grading that an orbit algebra inherits
from the levels of its repetitive category.  If the orbit algebra is
basic then such a derivation can be chosen to vanish on a given
complete set of primitive orthogonal idempotents and to have trace
half the dimension of the algebra.  On $A$ and $S$, however, every
derivation that vanishes on the trivial paths has trace zero.

\section{The algebras \texorpdfstring{$A$}{A} and \texorpdfstring{$S$}{S}}
\label{sec:algebras}

All algebras are $\kk$-algebras and all modules are finite-dimensional
right modules.  Paths are composed from left to right.  We write
$\dual$ for the duality functor $\Hom_{\kk}(-,\kk)$.

The algebra $A$ is given by the quiver
\begin{equation*}
  \begin{tikzcd}[column sep=4em,
      every arrow/.append style={shift left=0.75ex},
      every label/.append style={font=\normalsize}]
    1 \arrow[r,"b"] & 2 \arrow[l,"y"] \arrow[r,"d"] & 3 \arrow[l,"n"]
  \end{tikzcd}
\end{equation*}
modulo the relations
\begin{equation*}
  \begin{gathered}
    byb=bdnybdn,\qquad yby=dnybdny,\\
    ndn=nybdnyb,\qquad dnd=ybdnybd,\\
    bybd=ndny=0.
  \end{gathered}
\end{equation*}
In characteristic two the same presentation defines the
quaternion-type algebra $Q(3\mathcal{A})^{2}_{2}$, which \citet{BM}
found to be $3$-representation-finite.

The algebra $S$ is given by the quiver
\begin{equation*}
  \begin{tikzpicture}[baseline=(current bounding box.center),
      v/.style={inner sep=1.5pt,outer sep=2pt},
      el/.style={auto,swap,inner sep=2pt}]
    \node[v] (a1) at (90:2.4) {$a_1$};
    \node[v] (b1) at (150:2.4) {$b_1$};
    \node[v] (a2) at (210:2.4) {$a_2$};
    \node[v] (b2) at (270:2.4) {$b_2$};
    \node[v] (a3) at (330:2.4) {$a_3$};
    \node[v] (b3) at (30:2.4) {$b_3$};
    \node[v] (d1) at (150:1.2) {$d_1$};
    \node[v] (d2) at (270:1.2) {$d_2$};
    \node[v] (d3) at (30:1.2) {$d_3$};
    \draw[->] (a1) -- node[el] {$g_1$} (b1);
    \draw[->] (b1) -- node[el] {$s_1$} (a2);
    \draw[->] (a2) -- node[el] {$g_2$} (b2);
    \draw[->] (b2) -- node[el] {$s_2$} (a3);
    \draw[->] (a3) -- node[el] {$g_3$} (b3);
    \draw[->] (b3) -- node[el] {$s_3$} (a1);
    \draw[->] (a1) -- node[el,pos=.65] {$r_3$} (d3);
    \draw[->] (d3) -- node[el,pos=.35] {$q_3$} (a3);
    \draw[->] (a3) -- node[el,pos=.65] {$r_2$} (d2);
    \draw[->] (d2) -- node[el,pos=.35] {$q_2$} (a2);
    \draw[->] (a2) -- node[el,pos=.65] {$r_1$} (d1);
    \draw[->] (d1) -- node[el,pos=.35] {$q_1$} (a1);
  \end{tikzpicture}
\end{equation*}
modulo the relations
\begin{equation*}
  \begin{gathered}
    g_{1}s_{1}r_{1}=r_{3}q_{3}r_{2}q_{2}r_{1},\qquad
    r_{1}q_{1}g_{1}=g_{2}s_{2}g_{3}s_{3}g_{1},\\
    q_{1}g_{1}s_{1}=q_{1}r_{3}q_{3}r_{2}q_{2},\qquad
    s_{1}r_{1}q_{1}=s_{1}g_{2}s_{2}g_{3}s_{3},\\
    s_{1}g_{2}s_{2}r_{2}=q_{1}r_{3}q_{3}g_{3}=0
  \end{gathered}
\end{equation*}
and those obtained from them by cyclically shifting the indices in
$\ZZ_{3}$.

Both algebras are weighted surface algebras in the sense of
\citeauthor{ErdmannSkowronskiGV}~\cite{ErdmannSkowronskiGV,ErdmannSkowronskiCorr}.
Here $S$ is written on its Gabriel quiver, while the weighted surface
description of $A$ has virtual loops at $1$ and $3$~\cite{Erdmann}.
Since the relations are admissible, $A$ and $S$ are basic; their
dimensions are $36$ and $84$.  Each carries a non-degenerate symmetric
associative bilinear form and is therefore symmetric.  We verified the
forms---like the few other facts we use about $A$, $S$ and their
modules---\textit{in silico} with QPA~\cite{QPA} over $\QQ$.

\section{\texorpdfstring{$A$}{A} and \texorpdfstring{$S$}{S} are not orbit algebras}
\label{sec:orbit}\label{sec:nonorbit}

We begin by recalling the construction of orbit algebras.  Let
$\Lambda$ be a finite-dimensional algebra and fix a complete set of
its primitive orthogonal idempotents, indexed by $i$.  We define its
repetitive category $\rep$ as follows~\cite{HW}.  The set of objects
consists of one object $i[p]$ at level $p$ for each $i$ and each
$p\in\ZZ$.  By arranging the morphisms from level $p$ to level $q$ in
row $p$ and column $q$, we get the familiar matrix picture
\begin{equation*}
  \begin{pmatrix}
    \ddots & \ddots  &              &              \\
           & \Lambda & \dual\Lambda &              \\
           &         & \Lambda      & \dual\Lambda \\
           &         &              & \ddots
  \end{pmatrix}.
\end{equation*}
Morphisms within a level form a copy of $\Lambda$ on the diagonal,
those to the next level form a copy of $\dual\Lambda$ just above it,
and all other morphisms are zero.

The shift $\Nak(i[p])=i[p+1]$ is the Nakayama automorphism.  It is a
Serre functor, that is it satisfies
\begin{equation*}
  \dual\rep(x,y)\cong\rep(y,\Nak x)
\end{equation*}
naturally in $x,y$.  Conjugating $\Nak$ by a $\kk$-linear automorphism
$\varphi$ of $\rep$ gives another Serre functor.  It is a well-known
fact that a Serre functor is unique up to natural isomorphism, hence
$\varphi\Nak\cong\Nak\varphi$.

An automorphism $\varphi$ of $\rep$ is said to be admissible if no
non-zero power of $\varphi$ fixes the isomorphism class of an
indecomposable object.  Let $\varphi$ be admissible and let $\bar{x}$
denote the $\varphi$-orbit of an object $x$.  The morphism spaces of
the orbit category are given by
\begin{equation*}
  (\rep/\langle\varphi\rangle)(\bar x,\bar y) = \bigoplus_{j\in\ZZ}\rep(x,\varphi^{j} y).
\end{equation*}
The number of orbits, as well as the dimension of each direct sum on
the right, is finite~\cite[Lemma~3.4]{DI}.  By abuse of notation we
also denote by $\rep/\langle\varphi\rangle$ the algebra formed by the
direct sum of these morphism spaces.  We call this algebra the orbit
algebra.

\begin{proposition}\label{prop:orbit}
  Let $C$ be a basic finite-dimensional algebra with a complete set of
  primitive orthogonal idempotents $e_{v}$ and let $\Lambda$ be a
  finite-dimensional algebra with an admissible automorphism $\varphi$
  of $\rep$.  If $C\cong\rep/\langle\varphi\rangle$ then there is a
  derivation of $C$ which vanishes on every $e_{v}$ and has trace
  $\frac12\dim C$.
\end{proposition}

\begin{proof}
  To start with we make the level grading invariant under $\varphi$.
  Recall that by admissibility distinct objects in the same
  $\varphi$-orbit are not isomorphic.  On the other hand if two
  objects in different orbits were isomorphic then they would give
  isomorphic orthogonal idempotents in the basic algebra
  $\rep/\langle\varphi\rangle\cong C$.  Hence $\rep$ is skeletal.
  Since $\rep$ is skeletal the natural isomorphism
  $\varphi\Nak\cong\Nak\varphi$ implies that the two functors agree on
  objects.  Thus $\varphi$ induces a permutation on the finite set of
  $\Nak$-orbits $\setof{i[p]:p\in\ZZ}$.  By finiteness there exists an
  integer $m>0$ such that $\varphi^{m}$ maps each of these orbits to
  itself.  Let $s_{i}$ be such that $\varphi^{m}(i[0])=i[s_{i}]$.  Now
  note that $\varphi^{m}(i[p])=\Nak^{p}\varphi^{m}(i[0])=i[p+s_{i}]$
  for all $p$.

  If $\rep(i[p],j[q])\neq0$ then so are its images under the powers
  $\varphi^{mn}$, and as non-zero morphisms raise the level by $0$ or
  $1$, this gives $q-p+n(s_{j}-s_{i})\in\setof{0,1}$ for every
  $n\in\ZZ$.  Hence $s_{i}=s_{j}$, so $\varphi^{m}$ preserves level
  differences on non-zero morphism spaces.  Let $\ell(i[p])=p$ be the
  level and consider the averaged differences
  \begin{equation*}
    w(x,y)=\frac1m\sum_{r=0}^{m-1}
    \bigl(\ell(\varphi^{r}y)-\ell(\varphi^{r}x)\bigr).
  \end{equation*}
  They are $\varphi$-invariant on non-zero morphism spaces and satisfy
  $w(x,z)=w(x,y)+w(y,z)$.

  Now let us construct the grading.  Choose a set $\mathcal R$ of
  representatives of the $\varphi$-orbits of objects.  We put the
  summand $\rep(x,\varphi^{j}y)$ in degree $w(x,\varphi^{j}y)$ for
  $x,y\in\mathcal R$.  For $f\colon x\to\varphi^{i}y$ and
  $g\colon y\to\varphi^{j}z$, the product in
  $C'=\rep/\langle\varphi\rangle$ is
  $\varphi^{i}(g)\circ f\colon x\to\varphi^{i+j}z$, and by the
  invariance and additivity of $w$, its degree is the sum of their
  degrees.  These degrees give a rational grading of $C'$, and
  multiplication by degree is a derivation $\delta$ of $C'$.

  Let us compute the trace of $\delta$ by Serre duality.  For
  $x\in\mathcal R$ let $z_{x}\in\mathcal R$ and $t_{x}$ be such that
  $\Nak x=\varphi^{t_{x}}z_{x}$.  Since $\Nak$ permutes the
  $\varphi$-orbits, it follows that $x\mapsto z_{x}$ permutes
  $\mathcal R$.  We have the following isomorphism by Serre duality
  \begin{equation*}
    \dual\rep(x,\varphi^{j}y)\cong\rep(y,\varphi^{t_{x}-j}z_{x}).
  \end{equation*}
  Moreover, the degrees of these spaces add up to $w(x,\Nak x)=1$
  whenever the spaces are non-zero.  This follows since the degrees
  are $\varphi$-invariant and $\varphi^{r}\Nak x=\Nak\varphi^{r}x$ for
  every $r$.  By this isomorphism the spaces of $C'$ paired by the
  bijection $(x,y,j)\mapsto(y,z_{x},t_{x}-j)$ are equal in dimension,
  but have complementary degrees $q$ and $1-q$.  This shows that the
  degrees of a homogeneous basis of $C'$ are symmetric about
  $\frac12$.  The sum $\tr\delta$ of these degrees is then immediately
  seen to be $\frac12\dim C'$.

  Without loss of generality we may assume that $C=C'$.  Let
  $a=\sum_{v}\delta(e_{v})e_{v}$.  We have $[a,e_{v}]=\delta(e_{v})$
  by applying the Leibniz rule to the products $e_{u}e_{v}$.
  Replacing $\delta$ by $\delta-[a,-]$ we get a derivation which
  vanishes on every $e_{v}$.  Since $e_{v}\delta(e_{v})e_{v}=0$ and
  $C$ is basic, $a=\sum_{u\neq v}e_{u}\delta(e_{v})e_{v}$ lies in
  $\rad C$.  Left and right multiplication by $a$ are therefore
  nilpotent, so $\tr[a,-]=0$ and the trace is unchanged.
\end{proof}

It remains only to show that $A$ and $S$ have no such derivation.

\begin{proposition}\label{prop:trace}
  If a derivation of $A$ or of $S$ vanishes on the trivial paths then
  it has trace zero.
\end{proposition}

\begin{proof}
  Suppose that $C$ is $A$ or $S$ and that $\delta$ is a derivation of
  $C$ which vanishes on the trivial paths $e_{v}$.  The derivation
  $\delta$ is determined by its values on the arrows, and by the
  Leibniz rule we have $\delta(\alpha)=e_{v}\delta(\alpha)e_{w}$
  whenever $\alpha\colon v\to w$.  Conversely given such arrow values
  we define a derivation from the path algebra into $C$ by the Leibniz
  rule.  It descends to $C$ exactly when it sends every defining
  relation to zero.  We have a system of linear equations with
  rational coefficients by writing these conditions in a basis of
  paths.  An exact computation over $\QQ$ in QPA~\cite{QPA} shows that
  every solution has trace zero.  Moreover, all the coefficients are
  rational, so it follows that the solutions over $\kk$ are spanned by
  those over $\QQ$.
\end{proof}

As $A$ and $S$ are basic, and half their dimensions, $18$ and $42$,
are non-zero in characteristic zero, we get the second assertion of
Theorem~\ref{thm:main} from Propositions~\ref{prop:orbit}
and~\ref{prop:trace}.

\section{Constructing the cluster-tilting modules}
\label{sec:RF}

Now let us prove the first assertion of Theorem~\ref{thm:main}.  By
Iyama's description of cluster-tilting modules through their
endomorphism rings~\cite[Theorem~5.1(3)]{IyamaAC}, a
generator--cogenerator $M$ with $\Ext^{1}(M,M)=\Ext^{2}(M,M)=0$ is
$3$-cluster-tilting if and only if $\gldim\End(M)\le4$.  To use this
criterion we must find $M$ first, and then both conditions can be
checked by a finite computation.

For a vertex $v$ we denote by $P_{v}$ and $S_{v}$ the indecomposable
projective and the simple module at $v$, and by $\Omega X$ the kernel
of a projective cover of $X$.

Our starting point is the module $M_{0}$ that \citet{Erdmann}
considers for each algebra,
\begin{equation*}
  \begin{aligned}
    M_{0}                   & = A \oplus S_{2} \oplus \Omega^{2} S_{1}
    \oplus \Omega^{2} S_{3} &                                          & \text{for } A, \\
    M_{0}                   & = S\oplus\bigoplus_{i=1}^{3}
    \bigl(S_{a_{i}}\oplus\Omega^{2}S_{b_{i}}\oplus\Omega^{2}S_{d_{i}}\bigr)
                            &                                          & \text{for } S.
  \end{aligned}
\end{equation*}
The second syzygies in $M_{0}$, which we call the arms, are uniserial
quotients of projectives,
\begin{equation*}
  \begin{aligned}
    \Omega^{2}S_{1}       & \cong P_{2}/yA,
                          & \Omega^{2}S_{3}           & \cong P_{2}/dA,         \\
    \Omega^{2}S_{d_{i-1}} & \cong P_{a_{i}}/r_{i-1}S,
                          & \Omega^{2}S_{b_{i}}       & \cong P_{a_{i}}/g_{i}S,
  \end{aligned}
\end{equation*}
with composition factors from top to socle at vertices
$(2,3,2,1,2,3,2)$ and $(2,1,2,3,2,1,2)$ for $A$, and at
$(a_{i},b_{i},a_{i+1},b_{i+1},a_{i+2})$ and
$(a_{i},d_{i-1},a_{i-1},d_{i-2},a_{i-2})$ for $S$.  We call the vertex
$2$ of $A$ and the vertices $a_{1},a_{2},a_{3}$ of $S$ the waist.
Along an arm the composition factors alternate between waist vertices
and the other vertices, beginning and ending at the waist.

Next we will look for the remaining summands among subquotients of the
arms with top and socle at waist vertices.  For $A$ we denote by $U,V$
the quotients of $\Omega^{2}S_{3}$ of lengths $3$ and $5$, and by
$U',W$ the corresponding quotients of $\Omega^{2}S_{1}$.  For $S$ let
$B_{i}$ and $D_{i}$ denote the quotients of length $3$ of
$\Omega^{2}S_{d_{i-1}}$ and $\Omega^{2}S_{b_{i}}$ respectively, where
$i\in\ZZ_{3}$.  Up to isomorphism these are the only such subquotients
apart from the waist simples and the arms.  Moreover, any module in
the above list is also a submodule of an arm.  There are no $\Ext^{1}$
or $\Ext^{2}$ between $M_{0}$ and any of these modules in either
direction.  On the other hand \citeauthor{Erdmann} observed that
adjoining all of them gives self-extensions, and from this she
concluded that $M_{0}$ is not a summand of any $3$-cluster-tilting
module.

But we need not adjoin all of them, and Theorem~\ref{thm:RF} below
shows that this conclusion does not hold.  Any two of the four modules
for $A$ have a non-zero $\Ext^{1}$ between them in one direction or
the other, and so do any two of the $B_{i}$ and any two of the
$D_{i}$.  Between a $B_{i}$ and a $D_{j}$ there are no extensions of
degree $1$ or $2$.  We therefore test one new summand for $A$ and one
from each family for $S$.

\begin{theorem}\label{thm:RF}
  For each $X\in\setof{U,U',V,W}$ the $A$-module $M_{0}\oplus X$ is
  $3$-cluster-tilting, and so is the $S$-module
  $M_{0}\oplus B_{i}\oplus D_{j}$ for all $i,j\in\ZZ_{3}$.  Hence $A$
  and $S$ are $3$-representation-finite.
\end{theorem}

\begin{proof}
  Let $M$ be any of these modules.  Since $M$ contains the
  self-injective algebra $A$ or $S$ as a summand, it is a
  generator--cogenerator.  By Iyama's criterion it remains to check
  that $\Ext^{1}(M,M)=\Ext^{2}(M,M)=0$ and that $\End(M)$ has global
  dimension at most $4$.

  The author was too much of a nincompoop to do this legwork by hand
  or in his head, and thus we computed both over $\QQ$ with QPA, in
  exact arithmetic.  This gives vanishing of the two $\Ext$-groups
  and, from minimal right $\add M$-approximations, a projective
  resolution of length at most $4$ for each simple $\End(M)$-module.
  For instance $\End(M)$ has dimension $97$ for $M=M_{0}\oplus U$ and
  $219$ for $M=M_{0}\oplus B_{1}\oplus D_{1}$, and in both cases its
  global dimension is exactly $4$.  The program
  Auslander~\cite{Auslander} independently gives the same dimensions
  of $\Hom$- and $\Ext$-groups over several prime fields.  Since
  $\Hom$- and $\Ext$-groups commute with extending the field and
  global dimension is preserved by extending the perfect field
  $\QQ$~\cite[Lemma~1.1(2)]{BM}, it follows that the results over
  $\QQ$ hold over $\kk$ as well.
\end{proof}

\bibliographystyle{plainnat}
\bibliography{references}

\end{document}